\documentclass[11pt,a4paper]{article}
\usepackage{amsmath}
\usepackage{amssymb}
\usepackage{amsthm}
\usepackage{mathrsfs}
\usepackage{graphicx}
\usepackage{subfigure}
\usepackage{float}
\usepackage{setspace}
\usepackage{extarrows}
\usepackage{hyperref}
\usepackage{enumitem}
\usepackage{bm}
\usepackage{dsfont}
\usepackage{color}
\usepackage{xcolor}
\usepackage[a4paper,centering]{geometry}
\hypersetup{colorlinks=true,linkcolor=red,anchorcolor=blue,citecolor=blue}
\numberwithin{equation}{section}
\newtheorem{defi}{\textbf{Definition}}[section]
\newtheorem{lemma}[defi]{\textbf{Lemma}}

\newtheorem{thm}[defi]{\textbf{Theorem}}
\newtheorem{prop}[defi]{\textbf{Proposition}}
\newtheorem{rmk}[defi]{\textbf{Remark}}

\newcommand{\diam}{\mathrm{diam}}

\title{Tail exponents of the three-dimensional uniform spanning tree and Abelian sandpile}
\author{Xinyi Li\thanks{Beijing International Center for Mathematical Research, Peking University.\\ \href{mailto:xinyili@bicmr.pku.edu.cn}{xinyili@bicmr.pku.edu.cn}} \and Runsheng Liu\thanks{School of Mathematical Sciences, Peking University. \href{mailto:liurunsheng@pku.edu.cn}{liurunsheng@pku.edu.cn}} \and Daisuke Shiraishi\thanks{Graduate School of Informatics, Kyoto University. \href{mailto:shiraishi@acs.i.kyoto-u.ac.jp}{shiraishi@acs.i.kyoto-u.ac.jp}}}
\date{\today}

\begin{document}

\maketitle
\begin{abstract}
We study the local geometry of the three-dimensional uniform spanning tree and its connection with the Abelian sandpile model. We obtain sharp tail exponents, up to subpolynomial errors, for the past of the origin in the three-dimensional UST and for the $0$-tree of the $0$-wired uniform spanning forest. As a principal application, we prove the corresponding three-dimensional Abelian sandpile avalanche exponents: the avalanche-cluster radius has tail exponent $1$, while both the avalanche-cluster size and the total number of topplings have tail exponent $1/3$. These results identify the leading power-law behaviour of three-dimensional sandpile avalanches and improve previously known bounds.

\end{abstract}

\section{Introduction}
The \emph{uniform spanning tree} (UST) on a finite graph $G$ is the uniformly chosen random spanning tree of $G$. The study of the UST on $\mathbb{Z}^d$, which is defined as the weak limit of the UST on finite subgraphs, has been a central theme in modern probability theory, owing to its deep connections with random walks, loop-erased random walks (LERW), Abelian sandpile model, potential theory, and many other models in statistical mechanics.

The pioneering work of Pemantle \cite{Pem91} established the existence of the UST measure on $\mathbb{Z}^d$ for all $d\geq 2$, and proved that the limit is almost surely a single tree for $d\leq 4$ and a single-ended forest for $d\geq 5$. Wilson's algorithm \cite{Wil96}, which generates the UST by successively running loop-erased random walks, provides an explicit and powerful coupling between the UST and LERW. This coupling has been the basis for a large part of quantitative results on the geometry of these objects.

In two dimensions, the scaling limit of the UST is intimately related to the Schramm--Loewner evolution. Schramm \cite{Sch00} constructed the full scaling limit of 2D UST, and conjectured that the limit is conformally-invariant. Lawler, Schramm, and Werner \cite{LSW04} confirmed this conjecture by showing that the scaling limit of the Peano curve associated with the UST is $\mathrm{SLE}_8$. 
In three dimensions, building upon the groundbreaking work of Kozma \cite{Koz07}, which established the scaling limit of three-dimensional LERW,  Angel, Croydon, Hern\'andez-Torres, and Shiraishi \cite{ACHS21},  established the the scaling limit of the 3D UST and studied the associated random walk, establishing detailed metric and spectral properties of the limiting space. 

One of the main purposes of this work is to investigate the local geometry of 3D UST. In particular, we focus on the past of the origin and the $0$-tree of
the $0$-wired uniform spanning forest ($0$-WUSF), and obtain sharp tail exponents for their extrinsic diameter, intrinsic diameter, and volume. These estimates quantify the local non-mean-field geometry of the three-dimensional UST.

The second model we consider in this work is the \emph{Abelian sandpile model} (ASM), one of the canonical examples of self-organized criticality. On a finite set $K\subseteq\mathbb{Z}^{d}$, a sandpile configuration assigns a non-negative integer height to each vertex. A vertex is unstable if its height is at least $2d$, in which case it topples by sending one grain of sand to each neighbour; grains sent outside $K$ are lost. Dhar's Abelian property \cite{Dha90} states that stabilization does not depend on the order of topplings. Repeatedly adding a grain and stabilizing defines a Markov chain whose stationary states are the recurrent configurations. If one starts from stationarity and adds a grain at a fixed vertex, the ensuing topplings form an \emph{avalanche}; its radius, cluster, and total number of topplings are among the basic observables of the model. The burning bijection of Majumdar and Dhar \cite{MD92}, together with its infinite-volume extensions \cite{JR08,BHJ17}, identifies recurrent sandpiles with spanning trees or forests and makes UST/USF geometry a natural input for avalanche estimates. 

The Abelian sandpile model has been extensively studied through this tree representation. In high dimensions, Hutchcroft \cite{Hut20} proved mean-field avalanche exponents for $d\geq5$, while in the critical dimension $d=4$ Hutchcroft and Sousi \cite{HS23} identified non-trivial logarithmic corrections in the corresponding spanning-forest geometry. In low dimensions the situation is substantially more delicate. In particular, for $d=2,3$, Bhupatiraju, Hanson, and J\'arai \cite{BHJ17} obtained important bounds on avalanche exponents, but the exact leading powers were not known in dimension three.

The second main result of the present paper is to identify these leading powers in three dimensions. We prove that the avalanche-cluster radius has tail exponent $1$, while the avalanche-cluster size and the total number of topplings have tail exponent $1/3$, up to subpolynomial losses. Thus dimension three exhibits non-mean-field geometry, but the sandpile exponents are nevertheless determined by a precise characterization of the first wave via
the burning bijection (see Section \ref{subsec:asm} for details) by the $0$-tree of the $0$-WUSF. Its extrinsic and volume tails are the key inputs for the three-dimensional sandpile exponents.

\medskip

We now state our main results in a more precise fashion. Let $\mathcal{U}$ be the UST on $\mathbb{Z}^{3}$, and let $\mathfrak{P}$ denote the past of the origin. Let $\alpha\in[1/3,1)$ be the intersection exponent of 3D LERW; see \eqref{eq:alphadef}. We refer readers to Section~\ref{sec:notation} for the notation used below.

\begin{thm}[Extrinsic exponent]\label{thm:extrinsic}
There exist $c_{1}>0$, and for any $\varepsilon>0$,  $c_{2}=c_{2}(\varepsilon)>0$, such that for any $R\geq 1$,
\[
c_{1}R^{-\gamma}\leq\mathbb{P}\bigl(\diam_{\mathrm{ext}}(\mathfrak{P})\geq R\bigr)\leq c_{2}R^{-\gamma+\varepsilon},
\]
where $\gamma:=1+\alpha$.
\end{thm}

\begin{thm}[Intrinsic exponent]\label{thm:intrinsic}
For any $\varepsilon>0$, there exist $c_{1}=c_{1}(\varepsilon),c_{2}=c_{2}(\varepsilon)>0$, such that for any $R\geq 1$,
\[
c_{1}R^{-\eta-\varepsilon}\leq\mathbb{P}\bigl(\diam_{\mathrm{int}}(\mathfrak{P})\geq R\bigr)\leq c_{2}R^{-\eta+\varepsilon},
\]
where $\eta:=(1+\alpha)/(2-\alpha)$.
\end{thm}

\begin{thm}[Volume exponent]\label{thm:volume}
For any $\varepsilon>0$, there exist $c_{1}=c_{1}(\varepsilon),c_{2}=c_{2}(\varepsilon)>0$ such that for every $n\geq 1$,
\[
c_{1}n^{-\gamma/3-\varepsilon}\leq \mathbb{P}\bigl(|\mathfrak{P}|\geq n\bigr)\leq c_{2}n^{-\gamma/3+\varepsilon}.
\]
\end{thm}

In fact, with some additional work, the $-\varepsilon$ in both lower bounds of Theorems~\ref{thm:intrinsic} and \ref{thm:volume} can be removed; see Remarks~\ref{rmk:intrinsic} and \ref{rmk:volume} for relevant discussions.

\begin{rmk}\label{rmk:2d-exponents}
The same strategy also yields the corresponding exponents in two dimensions, which are recorded in Table~\ref{tab:past-exponents}. See Remark~\ref{rmk:discussion2D} for further discussion.
\end{rmk}

\begin{table}[H]
\centering
\small
\renewcommand{\arraystretch}{1.15}
\begin{tabular}{|c|c|c|c|}
\hline
Dimension  & Extrinsic & Intrinsic & Volume \\
\hline
$d=2$  & $\frac34$ & $\frac35$ & $\frac38$ \\
$d=3$  & $1+\alpha\, (\approx 1.376) $ & $\frac{1+\alpha}{2-\alpha}\,(\approx 0.847)$ & $\frac{1+\alpha}{3}\,(\approx 0.459)$ \\
$d=4$  & $2^\dagger$ & $1^\dagger$ & $\frac12^\dagger$ \\
$d\geq 5$  & $2$ & $1$ & $\frac12$ \\
\hline
\end{tabular}
\caption{Leading power exponents for the tail of the past. The approximate {\it numerical} values in dimension three come from Wilson's simulations \cite{Wil10}, which suggest $\alpha\approx0.376$. Inserting the rigorous bound $\alpha \in[1/3,1)$ gives $[4/3,2)$, $[4/5,2)$, $[4/9,2/3)$ respectively.  For $d=4$ the leading powers are accompanied by non-trivial logarithmic corrections (and hence the $\dagger$); see Hutchcroft--Sousi \cite{HS23}. The mean-field exponents for $d\geq 5$ follow from Hutchcroft \cite{Hut20}.}
\label{tab:past-exponents}
\end{table}

The proof of Theorems~\ref{thm:extrinsic}--\ref{thm:volume} has two main ingredients. The first is a one-point estimate for the past, obtained from Wilson's algorithm by expressing the event $\{x\in\mathfrak{P}\}$ in terms of a non-intersection event between an infinite LERW and an independent simple random walk. The second is a multiscale Wilson-algorithm construction on a spatial mesh, which turns the one-point estimate into a sharp upper bound for the extrinsic diameter tail.

The intrinsic and volume tails are then obtained by combining the extrinsic estimate with known metric-comparison results for the three-dimensional UST. More precisely, once the extrinsic scale is identified, the intrinsic scale follows from stretched-exponential control comparing Euclidean and tree distances, while the volume tail is proved by combining these metric comparisons with first- and second-moment estimates for the past inside a box.

Exponents for the $0$-tree can be obtained in a similar fashion. In this case, the one-point estimate is no-longer supplied by that of a non-intersection event, but rather a hitting event of a simple random walk. As a result, the corresponding extrinsic, intrinsic, and volume exponents become $1$, $\frac{1}{2-\alpha}$ and $\frac{1}{3}$, respectively. We refer readers to Propositions~\ref{prop:extrinsic-WUSF}-\ref{prop:volume-WUSF} for precise statements.

\medskip

We now turn to the Abelian sandpile model. Let $H$ be a sample from the infinite-volume recurrent sandpile measure, whose law is denoted by $\mathbf{P}$. Let $\mathrm{Av}:=\mathrm{Av}(0;H)$ be the {\it total number of topplings} (TNT) created by adding one grain at the origin, and let $\mathrm{AvC}:=\mathrm{AvC}(0;H)$ be the corresponding avalanche {\it cluster};  see Section~\ref{subsec:asm} for precise definitions. 

We state the main Abelian sandpile consequences. These results identify the leading power-law exponents for three-dimensional avalanches. They improve the previously known bounds of Bhupatiraju--Hanson--J\'arai \cite{BHJ17}, and show that in dimension three the avalanche-cluster radius has exponent $1$, while both the avalanche-cluster size and the total number of topplings have exponent $1/3$.

\begin{thm}[Extrinsic exponent]\label{thm:sandpile-radius}
There exist $c_{1}>0$, and for any $\varepsilon>0$,  $c_{2}=c_{2}(\varepsilon)>0$, such that for any $R\geq 1$,
\[c_{1}R^{-1}\leq\mathbf{P}(\diam_{\mathrm{ext}}(\mathrm{AvC})\geq R)\leq c_{2}R^{-1+\varepsilon}.\]
\end{thm}
\begin{thm}[TNT and avalanche cluster size exponent]\label{thm:sandpile-size}
For any $\varepsilon>0$, there exist $c_{1}=c_{1}(\varepsilon),c_2=c_2(\varepsilon)>0$, such that for any $n\geq 1$,
\[c_{1}n^{-1/3-\varepsilon}\leq\mathbf{P}(|\mathrm{AvC}|\geq n)\leq\mathbf{P}(\mathrm{Av}\geq n)\leq c_{2}n^{-1/3+\varepsilon}.\]
\end{thm}

We summarize exponents for ASM in various dimensions in the following table. See Remark \ref{rem:2D} for relevant discussiones in 2D.
\begin{table}[H]
\centering
\small
\renewcommand{\arraystretch}{1.15}
\begin{tabular}{|c|c|c|}
\hline
Dimension  & Extrinsic &  TNT / cluster size  \\
\hline
$d=2$  & $\left[0,\frac34\right]$ & $\left[0,\frac38\right]$ \\
$d=3$ &   $1$ & $\frac13$ \\
$d=4$ &  $2^\dagger$ & $\frac12^\dagger$ \\
$d\geq 5$  & $2$ & $\frac12$ \\
\hline
\end{tabular}
\caption{(Bounds of) tail exponents for the  diameter and size of the avalanche cluster.  The $d=2$ row records the bounds of Bhupatiraju--Hanson--J\'arai \cite{BHJ17}. The $d=3$ row is established in this paper; the best previously known bounds from \cite{BHJ17} were $[\frac{1}{6},1+\alpha]$ and $[\frac{1}{18},\frac{1+\alpha}{3}]$, respectively.  In $d=4$, the leading powers are the mean-field values but are accompanied by non-trivial logarithmic corrections (hence the
†); the four-dimensional UST input is due to Hutchcroft--Sousi \cite{HS23}, while the precise logarithmic corrections for sandpiles remain open. The mean-field row $d\geq 5$ follows from Hutchcroft \cite{Hut20}.}
\label{tab:0wusf-sandpile-exponents}
\end{table}

The sandpile proofs start from tail estimates for the $0$-tree in the $0$-WUSF and then transfer them to avalanches through the generalized burning bijection and the wave decomposition \cite{IKP94,JR08}. The point used for the lower bounds is the following first-wave correspondence: in finite volume, every {\it non-trivial} first-wave event is counted by the same two-component spanning forests that define the $0$-tree, up to the deterministic factor $G_{K}(0,0)$; after taking the infinite-volume limit this gives, for example,
\begin{equation*}
\mathbf{P}\bigl(\diam_{\mathrm{ext}}(W_{1})\ge R\bigr)=G(0,0)\,\mathbb{P}_{0}\bigl(\diam_{\mathrm{ext}}(\mathfrak{T})\ge R\bigr),\qquad R\ge1,
\end{equation*}
and the analogous identity for $|W_{1}|\ge n$, $n\ge2$. Since the factor $G(0,0)$ is a positive finite constant in dimension three, it does not affect the exponent. The upper bounds combine radius control with the fact that the number of waves has integrable tails in dimension three.

The paper is organized as follows. Section~\ref{sec:Prelim} collects notation and background on LERW, UST/USF, the $0$-WUSF, and the Abelian sandpile model. In Section~\ref{sec:past} we prove the tail exponents for the past of the origin in the three-dimensional UST. Section~\ref{sec:0wusf} contains the corresponding estimates for the $0$-tree in the $0$-WUSF. Finally, Section~\ref{sec:sandpile} applies these estimates to the Abelian sandpile model and briefly discusses why current methods do not apply to the 2D sandpile.

\medskip

\noindent {\bf Acknowledgements:} XL and RL are supported by the National Key R\&D Program of China (No.~2021YFA1002700).  DS is supported by  JSPS Grant-in-Aid for Scientific Research (B) 22H01128
and 26K00609. The authors thank Jack Hanson, Tom Hutchcroft and Yuval Peres for comments on a preliminary version of this work.

\section{Preliminaries}\label{sec:Prelim}
\subsection{Notation}\label{sec:notation}
\noindent\textbf{Constants.} Throughout the paper, $a,c,c_{1},M,\ldots$ denote positive finite constants whose values may change from line to line unless explicitly stated otherwise. Any dependence on auxiliary parameters will be indicated at first appearance.

\medskip

\par\noindent\textbf{Sets and metrics.} For $x\in\mathbb{Z}^{d}$, let $\|x\|$ denote the $\ell^{\infty}$-norm of $x$, and define
\[
B(x,r):=\{y\in\mathbb{Z}^{d}:\|y-x\|\leq r\},\quad \partial B(x,r):=\{y\notin B(x,r): \exists z\in B(x,r)\text{ with } y\sim z\}.
\]
We write $B_{R}:=B(0,R)$ for short. For $A\subseteq\mathbb{Z}^{d}$, we define its extrinsic diameter by
\[
\diam_{\mathrm{ext}}(A):=\sup_{x,y\in A}\|x-y\|.
\]
Whenever $A$ is a connected subgraph of a tree, $\diam_{\mathrm{int}}(A)$ denotes its graph diameter in the intrinsic tree metric.

\medskip
\par\noindent\textbf{Asymptotic notation.} We write $f\lesssim g$, $f\gtrsim g$, and $f\asymp g$ if there exists a positive constant $C$ such that $f\leq Cg$, $f\geq C^{-1}g$, and both inequalities hold, respectively.

\medskip
\par\noindent\textbf{Stopping times.} For a time-parametrized path $\gamma=(\gamma(t))_{t\geq0}$ and a set $A\subseteq\mathbb{Z}^{d}$, let
\[
\tau_{\gamma}(A):=\inf\{t\geq0:\gamma(t)\in A\},
\]
with the convention that $\inf\varnothing=\infty$. If $A=\partial B(x,r)$, we write $\tau_{\gamma}(x,r)$, and if $x=0$ we write $\tau_{\gamma}(r)$ for short. When the path is clear from context, we suppress it from the notation.

\medskip
\par\noindent\textbf{Loop erasure of paths.}
If $\pi=[\pi_{0},\ldots,\pi_{m}]$ is a finite nearest-neighbour path, its chronological loop erasure $\operatorname{LE}(\pi)$ is defined by setting $t_{0}:=\max\{j:\pi_{j}=\pi_{0}\}$ and, recursively, $t_{i+1}:=\max\{j:\pi_{j}=\pi_{t_{i}+1}\}$ until $t_{i}=m$; the loop-erased path is the self-avoiding path $[\pi_{t_{0}},\pi_{t_{1}},\ldots,\pi_{m}]$. For an infinite path that visits every vertex only finitely many times, $\operatorname{LE}(\pi[0,\infty))$ is defined by the same chronological procedure; equivalently, its restriction to each finite set agrees with the loop erasure of a sufficiently long initial segment.
\subsection{The loop-erased random walk}
In this subsection, we will recall some estimates of loop-erased random walk. 
Let $S[0,\infty)$ be a simple random walk on $\mathbb{Z}^3$ started from $x\in\mathbb{Z}^3$, whose law we denote by $P_x$. Write $P=P_0$ for short. Let $\gamma:={\rm LE}(S[0,\infty))$ be the loop-erasure of $S$, which we refer to as the infinite loop-erased random walk (ILERW) in $\mathbb{Z}^3$.

We begin with the non-intersection probability of an ILERW and another independent simple random walk.
\begin{prop}[{\cite[Theorem 1.2]{LS19}}]\label{prop:lerw-utc}
Let $\gamma$ be an ILERW started from the origin. Let $S$ be an independent simple random walk started from the origin. Denote their joint law by $P$. Let $\tau_{R}$ be the first hitting time of $\partial B_R$ with respect to $S$. Then, there exists a constant 
\begin{equation}\label{eq:alphadef}
 \alpha\in\left[\frac13,1\right)   
\end{equation}
(which we refer to as the (two-arm) {\rm intersection exponent} of 3D LERW), such that
\[P\bigl(\gamma\cap S[1,\tau_{R}]=\varnothing\bigr)\asymp R^{-\alpha}.\]
\end{prop}
For later use, we also record a simple consequence of Proposition \ref{prop:lerw-utc} that allows us to deal with random walks from outside to inside.
\begin{lemma}\label{lem:lerw-utc-rev}
Let $\gamma$ be an infinite LERW started from the origin. Let $S$ be an independent simple random walk started from some point $y\in\partial B_R$, conditioned to hit the origin. Denote their joint law by $P'_y$. Let $\tau_{0}$ be the first hitting time of the origin with respect to $S$. Then,
\[P'_y\bigl(\gamma\cap S[0,\tau_{0}]=\{0\}\bigr)\asymp R^{-\alpha}.\]
\end{lemma}
We recall a version of the separation lemma, which states that, conditioned on the non-intersection event above, there is a good chance for the SRW and LERW to be well-separated near their endpoints. Let $\tau_{\gamma}(R)$ and $\tau_{S}(R)$ be the first hitting times of $\partial B_R$ by $\gamma$ and $S$, respectively, and define
\[
\begin{aligned}
s(R):=\sup\Big\{\delta\in[0,1]:\;&
\bigl(\gamma[0,\tau_{\gamma}(R)]\cup B(\gamma(\tau_{\gamma}(R)),\delta R)\bigr)\\
&\cap
\bigl(S[1,\tau_{S}(R)]\cup B(S(\tau_{S}(R)),\delta R)\bigr)=\varnothing
\Big\}.
\end{aligned}
\]
We say that $\gamma$ and $S$ are ``$\delta$-separated at the end at scale $R$'' if $s(R)\geq\delta$.
\begin{lemma}[{\cite[Lemma 6.5]{Shi18}}]\label{lem:separation}
Recall the setup of Proposition \ref{prop:lerw-utc}. There exists $c>0$, such that for any $R>0$, we have
\[P(s(R)>1/10|s(R)>0)\geq c.\]
\end{lemma}
Based on Lemma~\ref{lem:separation}, we can derive the following lemma, which will be useful when proving the lower bound in Theorem~\ref{thm:volume}.
\begin{lemma}\label{lem:two-point-decompose}
Suppose $4r<R$ and $x\in\partial B_R$, $y\in\partial B_r$. Suppose that $\gamma$ is an ILERW started from $0$. $S_{1}$ and $S_{2}$ are two independent SRWs started from $x$ and $y$. Write $P_{x,y}$ for their joint law. Let $E$ denote the intersection of the following events:
\begin{itemize}
\item[1.]\ $S_{1}$ hits $0$ and, writing $\tau_{1}$ for the first hitting time of $0$, $S_{1}[0,\tau_{1}]\cap\gamma=\{0\}$.
\item[2.]\ $S_{2}$ hits $\operatorname{LE}(S_1[0,\tau_1])$ before it intersects $\gamma$.
\end{itemize}
Then,
\begin{equation}\label{eq:Elb}
P_{x,y}(E)\gtrsim (\log r)^{-1}R^{-1-\alpha}.
\end{equation}
\end{lemma}

Before proving Lemma~\ref{lem:two-point-decompose}, we record a classical bound from potential theory.
\begin{lemma}\label{lem:hittinglb}
Fix $0<a<b<1$. There exists $c=c(a,b)>0$ such that, for every $r\geq2$ and $z\in\mathbb{Z}^{3}$,
\[
\inf_{\Gamma} P_{z}\bigl(S[0,\tau(\partial B(z,br))]\cap \Gamma\neq\varnothing\bigr)\geq \frac{c}{\log r},
\]
where the infimum is over all nearest-neighbour paths $\Gamma\subseteq B(z,br)$ connecting $\partial B(z,ar)$ to $\partial B(z,br)$.
\end{lemma}
\begin{proof}
This follows from the standard capacity estimate that every nearest-neighbour crossing of a three-dimensional annulus of radius comparable to $r$ has capacity at least $c r/\log r$, together with the Green-function lower bound $G_{B(z,br)}(z,w)\geq c/r$ for points $w$ in a fixed inner subannulus; see, for example, \cite[Chapter~2]{Law96}. The last-exit decomposition then gives the stated lower bound.
\end{proof}
\begin{proof}[Proof of Lemma~\ref{lem:two-point-decompose}]
Let $H^{0}$ denote the event that $S_{1}$ hits $0$ and $S_{1}[0,\tau_{1}]\cap\gamma=\{0\}$. Let $S'$ be the time-reversal of $S_{1}[0,\tau_{1}]$. Let $\tau'$ be the stopping times corresponding to $S'$. We consider the following events:
\par
($H^{1}$) $S'$ first hits $\partial B_r$ in $B(y,r/100)$.
\par
($H^{2}$) $S'[0,\tau'(r)]$ and $\gamma[0,\tau_{\gamma}(r)]$ are $1/10$-separated at the end.
\par
($H^{3}$) Both $S'[\tau'(r),\tau'(2r)]$ and $\gamma[\tau_{\gamma}(r),\tau_{\gamma}(2r)]$ stay in some well-chosen tubes.
\par
($H^{4}$) $S'[\tau'(r),\tau_{1}]\cap (B_r\backslash B(y,r/50))=\varnothing$ and $S'[\tau'(2r),\tau_{1}]\cap(B_{2r}\backslash B(S'(\tau'(2r)),r/50))=\varnothing$.
\par
By Lemma~\ref{lem:separation}, Harnack principle and standard estimates of simple random walk, we have
\[P_{x,y}(H^{1}\cap H^{2}\cap H^{3}\cap H^{4}|H^{0})\gtrsim 1.\]
On the event $\cap_{i=0}^4H^{i}$, the loop-erasure of $S_{1}$ must contain a crossing between $\partial B(y,r/50)$ and $\partial B(y,r/10)$. Then, by Lemma~\ref{lem:hittinglb}, the probability that $S^{2}$ intersects such crossing before leaving $B(y,r/10)$ is at least $c(\log r)^{-1}$. Finally, noting that $P_{x,y}(H^{0})\asymp R^{-1-\alpha}$, we complete the proof.
\end{proof}
Finally, we provide a hittability estimate of LERW. For a (random) continuous curve $\gamma$, let 
\[H_{x}^{\gamma}(r,s;a):=\left\{\forall y\in B(x,r),\ {P}_{y}\bigl(S[0,\tau(\partial B(x,s))]\cap\gamma=\varnothing\bigr)\leq(r/s)^{a}\right\}.\]
\begin{lemma}[{\cite[Lemma 3.3]{SS18}}]\label{lem:hittability}
Let $\gamma$ be an ILERW started from $x$. Then, for any $M\geq1$, there exist constants $a=a(M)>0$ and $c=c(M)>0$, such that for all $s>r>1$, we have
\[{P}_{x}(H_{x}^{\gamma}(r,s;a))\geq1-c(r/s)^{M}.\]
\end{lemma}

\subsection{UST/USF, \texorpdfstring{$0$}{0}-WUSF and Wilson's algorithm}\label{subsec:ust-usf}
For a finite connected graph $G=(V,E)$, the \textit{uniform spanning tree} (UST) is the uniform measure on spanning trees of $G$. Let $(K_{n})_{n\geq1}$ be an exhaustion of $\mathbb{Z}^{d}$ by finite connected sets, and let $K_{n}^{*}$ be the finite graph obtained by wiring $K_{n}^{c}$ into a single boundary vertex $\partial_{n}$. The weak limit of the USTs on $K_{n}^{*}$ is the \textit{wired uniform spanning forest} (WUSF), and on $\mathbb{Z}^{d}$ it coincides with the free uniform spanning forest \cite{Pem91,BLPS01}.  When $d\leq4$ this limit is almost surely connected, and we continue to call it the UST on $\mathbb{Z}^{d}$; when $d\geq5$ it is disconnected and we call it the uniform spanning forest  (USF). In all dimensions $d\geq 2$, every component is almost surely one-ended \cite{BLPS01}.

For $d\geq 3$, for a distinguished vertex $v\in\mathbb{Z}^{d}$, let $K_{n}^{*v}$ be the graph obtained from $K_{n}^{*}$ by additionally identifying $v$ with $\partial_{n}$. The weak limit of the USTs on $K_{n}^{*v}$ is the $v$-wired uniform spanning forest, denoted by $\mathcal{U}_{v}$. In particular, $\mathcal{U}_{0}$ will denote the $0$-WUSF. We write $\mathbb{P}_0$ for the corresponding law. We write $\mathfrak{T}$ for the connected component of $0$ in $\mathcal{U}_{0}$ viewed after separating the two roots, and use $d_{0}$ for the graph distance in $\mathfrak{T}$.

For the usual UST/USF, we write $B_{\mathrm{int}}(0,R)$ for the intrinsic ball of radius $R$ around $0$ and $U_{R}$ for the connected component of $0$ in $\mathcal{U}\cap B_R$. For the $0$-WUSF, we write
\[B_{\mathrm{int}}^{0}(0,R):=\{x\in\mathfrak{T}: d_{0}(0,x)\leq R\},\]
and let $U_{R}^{0}$ be the connected component of $0$ in $\mathcal{U}_{0}\cap B_R$.

If $\mathcal{U}$ is a UST/USF on $\mathbb{Z}^d$, $d\geq 2$,  we write $\mathfrak{P}$ for the \emph{past} of the origin, namely the set of vertices whose unique\footnote{Thanks to the a.s.\ one-endedness of (trees of) $\mathcal{U}$.} path to infinity in $\mathcal{U}$ passes through the origin; when $\mathcal{U}$ is connected, this coincides with $\{0\}$ together with all finite component(s) separated from infinity by deleting the origin. 

We will repeatedly use Wilson's algorithm. On a finite graph with a designated root $r$, we enumerate the remaining vertices as $V\setminus\{r\}=\{x_{1},\ldots,x_{m}\}$, set $T_{0}=\{r\}$, and then proceed inductively: if $x_{i}\notin T_{i-1}$, we run a simple random walk $S_{i}$ from $x_{i}$ until its first hit on $T_{i-1}$, and set
\[T_{i}:=T_{i-1}\cup \mathrm{LE}(S_{i}),\]
where $\mathrm{LE}(S_i)$ stands for the chronological loop erasure of $S_i$.
Wilson \cite{Wil96} proved that the resulting tree has the law of UST. Benjamini, Lyons, Peres and Schramm \cite{BLPS01} extended this to transient infinite graphs by rooting at infinity; the same construction, rooted at $\{0,\infty\}$, yields the $0$-WUSF.

It follows from Wilson's algorithm and classical potential theory that for any $x\in\mathbb{Z}^{3}\backslash\{0\}$, 
\begin{equation}\label{eq:one-point-vWUSF}
\mathbb{P}_0(x\in\mathfrak{T})\asymp\|x\|^{-1}.    
\end{equation} 
The next three lemmas are standard comparisons between the intrinsic and extrinsic geometries of the three-dimensional UST, and will be used repeatedly in Sections~\ref{sec:past} and \ref{sec:0wusf}. For simplicity, we write 
\begin{equation}\label{eq:betadef}
\beta:=2-\alpha    
\end{equation}
 for the {\it growth exponent} of 3D LERW.
\begin{lemma}\label{lem:ext-int}
There exist  $c_{1},c_{2}>0$, such that for any $\lambda\geq1$ and $R\geq1$, 
\[\mathbb{P}(U_{R}\nsubseteq B_{\mathrm{int}}(0,\lambda R^{\beta}))\leq c_{1}e^{-c_{2}\lambda}.\]
\end{lemma}
This lemma can be proved in a similar manner as \cite[Proposition 3.5]{BM11}, and hence we omit the proof.

\begin{lemma}[{\cite[Proposition 6.1]{ACHS21}}]\label{lem:int-ext}
There exist $a,c_{1},c_{2}>0$, such that for any $\lambda>0$ and $R\geq1$, 
\[\mathbb{P}(B_{\mathrm{int}}(0,\lambda^{-1}R^{\beta})\nsubseteq B_R)\leq c_{1}e^{-c_{2}\lambda^{a}}.\]
\end{lemma}

\begin{lemma}[{\cite[Proposition 6.1]{ACHS21}}]\label{lem:int-vol}
There exist $a,c_{1},c_{2}>0$, such that for any $\lambda\geq1$ and $R\geq1$, 
\[\mathbb{P}(|B_{\mathrm{int}}(0,R)|\geq\lambda R^{3/\beta})\leq c_{1}e^{-c_{2}\lambda^{a}}.\]
\end{lemma}

For later use in the study of the $0$-WUSF, we also record the corresponding analogues for $\mathfrak{T}$. We omit the proofs as they can be proved in a similar manner.
\begin{lemma}\label{lem:ext-int-0}
There exist  $c_{1},c_{2}>0$, such that for any $\lambda\geq1$ and $R\geq1$, 
\[\mathbb{P}_{0}(U_{R}^{0}\nsubseteq B_{\mathrm{int}}^{0}(0,\lambda R^{\beta}))\leq c_{1}e^{-c_{2}\lambda}.\]
\end{lemma}

\begin{lemma}\label{lem:int-ext-0}
There exist $a,c_{1},c_{2}>0$, such that for any $\lambda>0$ and $R\geq1$, 
\[\mathbb{P}_{0}(B_{\mathrm{int}}^{0}(0,\lambda^{-1}R^{\beta})\nsubseteq B_R)\leq c_{1}e^{-c_{2}\lambda^{a}}.\]
\end{lemma}

\begin{lemma}\label{lem:int-vol-0}
There exist $a,c_{1},c_{2}>0$, such that for any $\lambda\geq1$ and $R\geq1$,
\[\mathbb{P}_{0}(|B_{\mathrm{int}}^{0}(0,R)|\geq\lambda R^{3/\beta})\leq c_{1}e^{-c_{2}\lambda^{a}}.\]
\end{lemma}

\subsection{Abelian sandpile model and burning bijection}\label{subsec:asm}
In this subsection, we briefly review the Abelian sandpile model and its relation to spanning trees and forests through the burning bijection.

Let $K\subseteq\mathbb{Z}^{d}$ be finite. A sandpile configuration on $K$ is a function $\eta:K\to\{0,1,2,\ldots\}$. The configuration is \emph{stable} at $v\in K$ if $\eta(v)<2d$, and it is stable if it is stable at every vertex. If $\eta(v)\geq 2d$, then $v$ is unstable and may be toppled, producing a new configuration $\eta'$ defined by
\[
\eta'(v)=\eta(v)-2d,
\qquad
\eta'(x)=
\begin{cases}
\eta(x)+1, & x\sim v,\\
\eta(x), & \text{otherwise}.
\end{cases}
\]
Thus, one toppling redistributes $2d$ grains from $v$ to its neighbours, and any grains sent outside $K$ are lost. Every unstable configuration stabilizes after finitely many topplings, and Dhar's Abelian property \cite{Dha90} states that the stabilized configuration does not depend on the order in which unstable vertices are toppled.

If one repeatedly adds a grain at a uniformly chosen vertex of $K$ and then stabilizes, one obtains a Markov chain on stable configurations. Its recurrent states are called the \emph{recurrent configurations}, and the stationary measure on recurrent configurations is denoted by $\nu_{K}$. Majumdar and Dhar \cite{MD92} proved that recurrent sandpiles are in bijection with wired spanning trees on $K$ via the burning bijection.

For $\eta\sim \nu_{K}$ and $v,x\in K$, let $N_{K}(v,x;\eta)$ denote the number of times that $x$ topples when one adds a grain at $v$ and stabilizes. We define the \emph{avalanche cluster} and the \emph{total number of topplings} by
\[
\mathrm{AvC}_{K}(v;\eta):=\{x\in K:N_{K}(v,x;\eta)\geq 1\},
\qquad
\mathrm{Av}_{K}(v;\eta):=\sum_{x\in K}N_{K}(v,x;\eta).
\]

In dimensions $d\geq 3$, Jar\'ai and Redig \cite{JR08} constructed the infinite-volume sandpile measure. Let $H$ denote a sample from this measure, and write $\mathbf{P}$ and $\mathbf{E}$ for its law and expectation.

\begin{lemma}\label{lem:asm-inf}
Suppose $d\geq 3$. Then the avalanche caused by adding one grain at any fixed vertex $v$ to $H$ is $\mathbf{P}$-almost surely finite. If $N(v,x)$ denotes the number of topplings at $x$, then
\[
\mathrm{AvC}(v;H):=\{x\in\mathbb{Z}^{d}:N(v,x)\geq 1\}
\]
and
\[
\mathrm{Av}(v;H):=\sum_{x\in\mathbb{Z}^{d}}N(v,x)
\]
are almost surely well defined. We abbreviate
\[
\mathrm{Av}:=\mathrm{Av}(0;H),\qquad \mathrm{AvC}:=\mathrm{AvC}(0;H).
\]
\end{lemma}

The decomposition of the avalanche into waves and a generalized burning algorithm is first established by Ivashkevich, Ktitarev and Priezzhev \cite{IKP94}. We will use the description in Jar\'ai and Redig \cite{JR08} for this. Starting from $H+\delta_{0}$, one topples the origin once and then topples every unstable vertex except the origin; the set of sites toppled in this stage is the first wave $W_{1}(H)$. If the origin is still unstable afterwards, one repeats the same procedure to obtain $W_{2}(H)$, and so on, until the origin becomes stable. Writing $N(0,0;H)$ for the number of waves, one has
\[
\mathrm{AvC}=\bigcup_{k=1}^{N(0,0;H)}W_{k}(H).
\]

We shall use the generalized burning construction of \cite[Section~7]{JR08} in a form that is slightly more precise than a bare distributional identification. For a finite set $K\Subset\mathbb{Z}^{d}$ with $0\in K$, put $K^{\circ}:=K\setminus\{0\}$. Let $R_{K}$ denote the recurrent configurations in $K$, and let $\nu_{K}^{(0)}$ be the uniform measure on two-component spanning forests of the wired graph over $K$, with one component rooted at $0$ and the other rooted at the wired boundary. Write $\mathfrak{T}_{K}$ for the vertex set of the component of $0$ under $\nu_{K}^{(0)}$. We use the convention that the first wave $W_{1,K}$ is empty if the addition at $0$ does not make $0$ topple.

\begin{lemma}\label{lem:burning}
Let $K\Subset\mathbb{Z}^{d}$ contain $0$ and all nearest neighbours of $0$. Let $\mathcal{A}$ be an event of subsets of $K$ such that every $A\in\mathcal{A}$ contains at least one nearest neighbour of $0$. Then
\begin{equation}\label{eq:finite-wave-tree-correspondence}
\nu_{K}\bigl(W_{1,K}\in\mathcal{A}\bigr)
= g_{K}\,\nu_{K}^{(0)}\bigl(\mathfrak{T}_{K}\in\mathcal{A}\bigr),
\qquad
 g_{K}:=\frac{|R_{K^{\circ}}|}{|R_{K}|}=G_{K}(0,0),
\end{equation}
where $G_{K}$ is the inverse of the finite-volume toppling matrix.
\end{lemma}

\begin{proof}
The burning construction in \cite[Section~7]{JR08} associates to each $\xi\in R_{K^{\circ}}$ a two-component spanning forest with roots $0$ and the wired boundary; its $0$-component has vertex set $\mathfrak{T}_{K}(\xi)$. Equivalently, $\mathfrak{T}_{K}(\xi)$ is the set burnt in the first phase when $0$ and the boundary are treated as the two roots. If $\mathfrak{T}_{K}(\xi)$ contains a nearest neighbour of $0$, then setting the height at $0$ to its maximal stable value $2d-1$ gives a recurrent configuration on $K$, and the first wave of this configuration is exactly $\mathfrak{T}_{K}(\xi)$. Conversely, every first wave that contains a nearest neighbour of $0$ arises uniquely in this way from the restriction of the configuration to $K^{\circ}$. Therefore
\[
\nu_{K}\bigl(W_{1,K}\in\mathcal{A}\bigr)
=\frac{1}{|R_{K}|}\,\#\{\xi\in R_{K^{\circ}}:\mathfrak{T}_{K}(\xi)\in\mathcal{A}\}
=\frac{|R_{K^{\circ}}|}{|R_{K}|}\,\nu_{K}^{(0)}\bigl(\mathfrak{T}_{K}\in\mathcal{A}\bigr).
\]
The determinant formula for recurrent configurations gives $|R_{K}|=\det\Delta_{K}$ and $|R_{K^{\circ}}|=\det\Delta_{K^{\circ}}$. By Cramer's rule,
$|R_{K^{\circ}}|/|R_{K}|=(\Delta_{K}^{-1})_{00}=G_{K}(0,0)$.
\end{proof}

Combining Lemma~\ref{lem:burning} with the infinite-volume limits in \cite[Sections~6--8]{JR08} yields the tail form of the non-trivial first-wave/$0$-tree correspondence needed below.

\begin{lemma}\label{lem:burning-inf}
Assume $d\ge3$. For every $R\ge1$ and every integer $n\ge2$,
\begin{align}
\mathbf{P}\bigl(\diam_{\mathrm{ext}}(W_{1}(H))\ge R\bigr)
&=G(0,0)\,\mathbb{P}_{0}\bigl(\diam_{\mathrm{ext}}(\mathfrak{T})\ge R\bigr),\label{eq:first-wave-radius-tree}\\
\mathbf{P}\bigl(|W_{1}(H)|\ge n\bigr)
&=G(0,0)\,\mathbb{P}_{0}\bigl(|\mathfrak{T}|\ge n\bigr).\label{eq:first-wave-volume-tree}
\end{align}
More generally, the same identity holds for every continuity event $\mathcal{A}$ of finite vertex sets such that every set in $\mathcal{A}$ contains a nearest neighbour of $0$.
\end{lemma}

\begin{proof}
Apply Lemma~\ref{lem:burning} with $K=B_{N}$. The finite sandpile measures $\nu_{B_{N}}$ converge to the infinite-volume sandpile measure $\mathbf{P}$, and the wave construction converges because avalanches are almost surely finite in dimensions $d\ge3$; see \cite[Sections~3 and~6]{JR08}. Moreover, \cite[Proposition~7.11]{JR08} gives $\nu_{B_{N}}^{(0)}\Rightarrow\mathbb{P}_{0}$ and finiteness of the limiting $0$-component, while $G_{B_{N}}(0,0)\to G(0,0)$. Passing to the limit in \eqref{eq:finite-wave-tree-correspondence} gives the asserted identity for cylinder continuity events. The events in \eqref{eq:first-wave-radius-tree} and \eqref{eq:first-wave-volume-tree} follow by first intersecting with the event that the relevant finite tree is contained in $B_{M}$ and then letting $M\to\infty$, using the almost-sure finiteness of both $W_{1}(H)$ and $\mathfrak{T}$.
\end{proof}

We will also use Dhar's formula for the expected number of topplings. In finite volume,
\begin{equation}\label{eq:Dhar}
  \int N_{K}(v,x;\eta)\,\nu_{K}(d\eta)=G_{K}(v,x),
\end{equation}
where $G_{K}(v,x)$ is the Green's function of simple random walk killed upon exiting $K$. Passing to the infinite-volume limit in dimensions $d\geq 3$ \cite{JR08}, we obtain
\begin{equation}\label{eq:Dhar-inf}
\mathbf{E}[N(v,x)]=G(v,x),
\end{equation}
where $G$ is the full-space Green's function on $\mathbb{Z}^{d}$.
\section{Estimates for the past of origin}\label{sec:past}
\subsection{Extrinsic estimates}
\begin{lemma}\label{lem:one-point}
For any $x\in\partial B_{R}$, we have
\[\mathbb{P}(x\in\mathfrak{P})\asymp R^{-(1+\alpha)}.\]
\end{lemma}
\begin{proof}
We consider the Wilson's algorithm started from the origin, and let $\gamma$ be the ILERW that corresponds to the first step of Wilson's algorithm. If $\{x\in\mathfrak{P}\}$ occurs, then obviously we have $x\notin\gamma$. Then we run the second step of Wilson's algorithm started from $x$ and denoted by $S$ the corresponding random walk path. Now, the key observation is that $x\in\mathfrak{P}$ if and only if $S$ is not an infinite path and stops at the origin. Therefore, $\mathbb{P}(x\in\mathfrak{P})$ is equivalent to the probability of the following events: let $\gamma$ be an infinite LERW started from the origin, and $S$ be an independent SRW started from $x$, which satisfies
\par
($E_{1}$)\ $S$ visits $0$;
\par
($E_{2}$)\ Let $\tau$ be the first hitting time of $0$ with respect to $S$, then $S[0,\tau]\cap\gamma=\{0\}$.
By classical potential theory,
\[\mathbb{P}(E_{1})\asymp R^{-1}.\]
By Lemma~\ref{lem:lerw-utc-rev}, we have
\[\mathbb{P}(E_{2}|E_{1})\asymp R^{-\alpha}.\]
Hence we get $\mathbb{P}(E_{1}\cap E_{2})\asymp R^{-1-\alpha}$, which completes the proof.
\end{proof}
Next, we are going to use a multiscale mesh to generate the UST via Wilson's algorithm. We will prove the upper bound of Theorem~\ref{thm:extrinsic} by showing that, with high probability, $\{\diam_{\mathrm{ext}}(\mathfrak{P})\geq R\}$ implies that one can discover some $x\in \mathfrak{P}\backslash B_R$ on the first layer of mesh. For $k\geq1$ set
\[r_{k}:=2^{-3k}R^{1-2\varepsilon},\]
and stop the construction at
\[
K_{R}:=\max\{k\geq1:r_{k}\geq1\}.
\]
For $1\leq k\leq K_{R}$, let $\Lambda_{k}$ be an increasing family of subsets of $\mathbb{Z}^{d}$ such that $\Lambda_{k}\backslash\Lambda_{k-1}\subseteq B(0,(1+2^{-k})R)$, $|\Lambda_{k}|\asymp(R/r_{k})^{3}=2^{9k}R^{6\varepsilon}$ and
\[\bigcup_{x\in\Lambda_{k}}B(x,r_{k})\supseteq B_{(1+2^{-k})R}.\]
We also define a mesoscopic step size
\[s_{k}:=2^{-2k}R^{1-\varepsilon}.\]
We now consider a increasing family of partial trees $\mathcal{T}_{k}$, which is generated by the following procedure. Let $\mathcal{T}_{0}$ be the partial tree obtained by the Wilson's algorithm at the origin, and let $\mathcal{T}_{1},\mathcal{T}_{2},\ldots$ be the partial trees obtained by running Wilson's algorithm at all points (in a deterministic order) in $\Lambda_{1},\Lambda_{2},\ldots,$ and so on. We now assign a colour (red or blue) to every vertex in the partial tree that indicates whether it belongs to the past of the origin. We first colour the origin red and $\mathcal{T}_{0}\backslash\{0\}$ blue. Then, at each step in the Wilson's algorithm, if the LERW dies at a red vertex, then it will be coloured red, otherwise it will be coloured blue. Note that a vertex belongs to the past of the origin if and only if it is red.
\par
For $1\leq k\leq K_{R}$, let $E_{k}$ be the event that there exists a red vertex on $\partial B_{(1-2^{-k})R}\cap \mathcal{T}_{k}$. For $2\leq k\leq K_{R}$, set
\[
F_{k}:=E_{k}\cap(E_{1}\cup\cdots\cup E_{k-1})^{c}.
\]

\begin{lemma}\label{lem:first-layer-dominate}
For large $R$, we have
\[\sum_{2\leq k\leq K_{R}}\mathbb{P}(F_{k})\leq R^{-99}.\]
\end{lemma}
\begin{proof}
On the event $F_{k}$, there exists $x\in\partial B_{(1-2^{-k})R}\cap \mathcal{T}_{k}$ such that $x$ is red, and we will consider the first one if there exist multiple red vertices (note that we perform the exploration process in a deterministic order). Note that there is no red vertex on $\partial B_{(1-2^{-k+1})R}\cap \mathcal{T}_{k-1}$. Then, while running the Wilson's algorithm started from some $y\in\Lambda_{k}\backslash\Lambda_{k-1}$ passing through $x$, the corresponding random walk path has to avoid all pre-existing blue trajectories. We also note that it takes at least distance $2^{-k}R$ for $x$ to connect to some red vertices that have already occurred in $\mathcal{T}_{k-1}$. 
\par
Denoted by $G$ the event that all the paths generated via Wilson's algorithm started from $\Lambda_{k-1}$ satisfies the corresponding event $H$ defined in Lemma~\ref{lem:hittability} in the mesoscopic scale. Applying Lemma~\ref{lem:hittability} and a union bound among all pre-existing paths connecting $x$ and $\infty$ for all blue points $x\in\Lambda_{k-1}$, we have
\[\mathbb{P}(G^{c})\leq 2^{9k}R^{6\varepsilon}(2^{k}R^{\varepsilon})^{-M},\]
where $M$ is a large constant to be determined. On the event $G$, we now consider the probability for the random walk $S_{x}$ started from $x$ to hit $\partial B_{(1-2^{-k+1})R}$ without touching the blue curves. For this, we note that $S_{x}$ passes through at least $2^{k}R^{\varepsilon}$ mesoscopic steps (a step means that $S_{x}$ enters $B(y,r_{k-1})$ and exits $B(y,s_{k-1})$ for some $y\in\Lambda_{k-1}$). Thus, by Lemma~\ref{lem:hittability}, this probability is at most $(2^{k}R^{\varepsilon})^{-a2^{k}R^{\varepsilon}}$. Therefore, we conclude that
\[\mathbb{P}(F_{k})\leq\mathbb{P}(G^{c})+\mathbb{P}(F_{k}\cap G)\leq 2^{9k}R^{6\varepsilon}(2^{k}R^{\varepsilon})^{-M}+(2^{k}R^{\varepsilon})^{-a2^{k}R^{\varepsilon}}.\]
For arbitrary $\varepsilon$, we take $M=1000\max\{1,\varepsilon^{-1}\}$ and the corresponding $a=a(M)$ is then determined via Lemma~\ref{lem:hittability}. We then take a large $R$  such that $\min\{1,\varepsilon\}aR^{\varepsilon}>100$. We then have
\[\mathbb{P}(F_{k})\leq 2^{-k}R^{-99}.\]
Summing over $k$, we complete the proof.
\end{proof}
We can now prove Theorem~\ref{thm:extrinsic}.
\begin{proof}[Proof of Theorem~\ref{thm:extrinsic}]
The lower bound is a direct consequence of Lemma~\ref{lem:one-point}. For the upper bound, we note that $\{\diam_{\mathrm{ext}}(\mathfrak{P})\geq R\}$ will imply $E_{1}\cup\bigcup_{2\leq k\leq K_{R}}F_{k}$. By Lemma~\ref{lem:one-point}, we have
\begin{equation}\label{eq:upper-bound-ext}
\mathbb{P}(E_{1})\leq c(\varepsilon)R^{-1-\alpha+\varepsilon}.
\end{equation}
By Lemma \ref{lem:first-layer-dominate}, we have
\[\mathbb{P}(\diam_{\mathrm{ext}}(\mathfrak{P})\geq R)\leq\mathbb{P}(E_{1})+\sum_{2\leq k\leq K_{R}}\mathbb{P}(F_{k})\leq c(\varepsilon)R^{-1-\alpha+\varepsilon}+R^{-99},\]
which completes the proof.
\end{proof}
\subsection{Intrinsic estimates and volume estimates}
In this subsection, we will prove Theorems~\ref{thm:intrinsic} and \ref{thm:volume}. We will repeatedly use the metric-comparison lemmas from Section~\ref{subsec:ust-usf}, namely Lemmas~\ref{lem:ext-int}, \ref{lem:int-ext}, and \ref{lem:int-vol}. Recall that $\mathcal{U}$ is the UST in $\mathbb{Z}^{d}$ and that $U_{R}$ denotes the connected component of $0$ in $\mathcal{U}\cap B_R$. 
\begin{proof}[Proof of Theorem~\ref{thm:intrinsic}.]
Fix $R$ large and let $\lambda>0$. Consider the events
\[A_{1}:=\{\diam_{\mathrm{int}}(\mathfrak{P})\geq R^{\beta}\},\qquad A_{2}:=\{\diam_{\mathrm{ext}}(\mathfrak{P})\geq \lambda R\}.
\]
For the upper bound, on $A_{1}\cap A_{2}^{c}$ there exists $x\in\mathfrak{P}$ with $d(0,x)\geq R^{\beta}/2$, while $A_{2}^{c}$ implies that $\mathfrak{P}\subseteq B_{\lambda R}$ and hence $\mathfrak{P}\subseteq U_{\lambda R}$. Therefore, by Lemma~\ref{lem:ext-int},
\[\mathbb{P}(A_{1}\cap A_{2}^{c})\leq \mathbb{P}(U_{\lambda R}\nsubseteq B_{\mathrm{int}}(0,R^{\beta}))\leq c_{1}e^{-c_{2}\lambda^{-\beta}}.
\]
Taking $\lambda=R^{-\varepsilon}$ and using \eqref{eq:upper-bound-ext}, we obtain
\[\mathbb{P}(A_{1})\leq \mathbb{P}(A_{2})+c_{1}e^{-c_{2}R^{\beta\varepsilon}}\leq c(\varepsilon)R^{(1-\varepsilon)(-1-\alpha+\varepsilon)},
\]
which gives the desired upper bound after renaming $R^{\beta}$ as the intrinsic scale.

For the lower bound, on $A_{1}^{c}\cap A_{2}$ there exists $x\in\mathfrak{P}\setminus B_{\lambda R}$, while $A_{1}^{c}$ implies that $x\in B_{\mathrm{int}}(0,R^{\beta})$. Hence Lemma~\ref{lem:int-ext} yields
\[\mathbb{P}(A_{1}^{c}\cap A_{2})\leq \mathbb{P}(B_{\mathrm{int}}(0,R^{\beta})\nsubseteq B_{\lambda R})\leq c_{1}e^{-c_{2}\lambda^{a\beta}}.
\]
Taking $\lambda=R^{\varepsilon}$ and using Lemma~\ref{lem:one-point}, we obtain
\[\mathbb{P}(A_{1})\geq \mathbb{P}(A_{2})-c_{1}e^{-c_{2}R^{a\beta\varepsilon}}\geq cR^{-(1+\alpha)(1+\varepsilon)}
\]
for all large $R$, which gives the claimed lower bound after the same change of variable.
\end{proof}
\begin{rmk}\label{rmk:intrinsic}
In fact, we can prove an up-to-constant lower bound of Theorem~\ref{thm:intrinsic} by showing that 
\[\mathbb{P}(d_{\mathcal{U}}(0,x)\geq cR^{\beta}|x\in\mathfrak{P})\geq c'\mbox{ for }x\in\partial B_{R}.\]
This can be done by first finding a segment from the two-sided LERW whose law is comparable to a ``clean'' LERW excursion, and then using \cite{Shi18} to control the intrinsic length of this segment. 
\end{rmk}
We now turn to the proof of Theorem~\ref{thm:volume}. The upper bound follows from Theorem~\ref{thm:extrinsic} and Lemma~\ref{lem:int-vol}. For the lower bound, we require the following lemma. 
\begin{lemma}\label{lem:conditional-moment}
    Suppose that $x\in\partial B_R$. There exist constants $c_{1},c_{2}>0$, such that for all $y,z\in B_{R/4}$, we have
\[\mathbb{P}(y\in\mathfrak{P}|x\in\mathfrak{P})\geq c_{1}(\log R)^{-1}.\]
\end{lemma}
\begin{proof}
The claims follow from Lemma~\ref{lem:one-point} and the following two-point estimate:
\[\mathbb{P}(x,y\in\mathfrak{P})\gtrsim (\log r)^{-1}R^{-1-\alpha}\quad \forall r<R,\  x\in\partial B_{r},\ y\in\partial B_{R},
\]
which is a direct consequence of Lemma~\ref{lem:two-point-decompose}.
\end{proof}
\begin{proof}[Proof of Theorem~\ref{thm:volume}.]
We begin with the upper bound and consider the following two events: $G_{1}:=\{|\mathfrak{P}|\geq R^{3/\beta}\}$ and $G_{2}:=\{\diam_{\mathrm{int}}(\mathfrak{P})\geq cR\}$. On $G_{1}\cap G_{2}^{c}$, we see that $B_{\mathrm{int}}(0,cR)\supseteq\mathfrak{P}$, and thus $|B_{\mathrm{int}}(0,cR)|\geq R^{3/\beta}$. By Lemma~\ref{lem:int-vol}, we have
\[\mathbb{P}(G_{1}\cap G_{2}^{c})\leq c_{1}e^{-c_{2}c^{-3a/\beta}}.\]
Fix $\varepsilon>0$, and let $c=R^{-\varepsilon}$. For large $R$, by Theorem~\ref{thm:intrinsic}, we have
\[\mathbb{P}(G_{1})\leq\mathbb{P}(G_{1}\cap G_{2}^{c})+\mathbb{P}(G_{2})\leq c_{1}e^{-c_{2}R^{3a\varepsilon/\beta}}+c'R^{-(1-\varepsilon)\eta+\varepsilon'},\]
which concludes the upper bound.
\par
We now turn to the lower bound. By Lemma~\ref{lem:conditional-moment}, for any $x\in\partial B_R$, we have
\[\mathbb{E}\big[|\mathfrak{P}\cap B_R|\big|x\in\mathfrak{P}\big]\geq c_{1}R^{3}(\log R)^{-1}.\]
Trivially,
\[\mathbb{E}\big[|\mathfrak{P}\cap B_R|^{2}\big|x\in\mathfrak{P}\big]\leq c_{2}R^{6}.\]
Hence, there exist  $c,c'>0$, such that
\[\mathbb{P}\big(|\mathfrak{P}\cap B_R|\geq cR^{3}(\log R)^{-1}\big|x\in\mathfrak{P}\big)\geq c'(\log R)^{-2}.\]
By Lemma~\ref{lem:one-point}, we have $\mathbb{P}(x\in\mathfrak{P})\asymp R^{-1-\alpha}$. Therefore,
\[\mathbb{P}(|\mathfrak{P}|\geq cR^{3}(\log R)^{-1})\geq cR^{-1-\alpha}(\log R)^{-2}.\]
Since logarithmic factors can be absorbed into a factor of $R^{\varepsilon}$, this yields the lower bound in Theorem~\ref{thm:volume} after reparametrizing the threshold.
\end{proof}
\begin{rmk}\label{rmk:volume}
We can remove all polylogarithmic correction terms in the proof of the lower bound above by a stronger version of Lemma~\ref{lem:two-point-decompose}, which removes the factor of $(\log r)^{-1}$ in \eqref{eq:Elb} and in turn provides an up-to-constant lower bound in Theorem~\ref{thm:volume}. However, such an improvement requires highly involved estimates for LERW, which would make this work quite heavy. This prompts us to provide with a slightly weaker version only, which still matches in principal exponent the upper bound we obtain. 
\end{rmk}

\begin{rmk}\label{rmk:discussion2D}
We now briefly discuss how one obtains the 2D exponents summarized in Table \ref{tab:past-exponents}. In this case, the exponent ``$1+\alpha$'' in Lemma~\ref{lem:one-point} will be replaced by $\alpha_{2}:=3/4$ (the two-arm intersection exponent for 2D LERW). The mesh argument follows similarly, which yields the extrinsic exponent. Then, the 2D counterparts of  Lemmas~\ref{lem:ext-int}--\ref{lem:int-vol}, the (stretched) exponential bounds for UST, established in \cite[Section 3]{BM11}, yield intrinsic estimates and volume estimates, where corresponding exponents are $\frac{\alpha_{2}}{2-\alpha_{2}}=3/5$ and $\alpha_{2}/2=3/8$, respectively.
\end{rmk}
\section{Estimates for the $0$-tree}\label{sec:0wusf}
In this section, we give in Propositions \ref{prop:extrinsic-WUSF}--\ref{prop:volume-WUSF} the tail exponents of the $0$-tree in 3D $0$-WUSF. In Section \ref{sec:sandpile}, we will use the extrinsic and volume exponents to deduce exponents for the 3D Abelian sandpile model. 

We use a multiscale mesh to generate the $0$-WUSF, similar to the construction in the previous section. Recall the quantities and objects $r_{k},\Lambda_{k},s_{k}$ and the stopping index $K_{R}$ from the previous section. We define an increasing family of partial forests $\mathcal{F}_{k}$ in a similar fashion: Let $\mathcal{F}_{0}$ be the partial forest with two marked points $0,\infty$. Let $\mathcal{F}_{1},\mathcal{F}_{2},\ldots$ be the partial trees obtained by running Wilson's algorithm at all points in $\Lambda_{1},\Lambda_{2},\ldots$, and so on. Similarly, every vertex is coloured either red or blue. In $\mathcal{F}_{0}$, $0$ is coloured red and $\infty$ is coloured blue. At each step, if the LERW is attached to a red vertex, then it will be coloured red, otherwise blue. Then, a vertex in $B_R$ is red if and only if it belongs to $\mathfrak{T}$.
\par 
For $1\leq k\leq K_{R}$, write $E_{k}$ for the event that there exists a red vertex on $\partial B_{(1-2^{-k})R}\cap \mathcal{F}_{k}$. For $2\leq k\leq K_{R}$, set
\[
G_{k}:=E_{k}\cap(E_{1}\cup\cdots\cup E_{k-1})^{c}.
\]

The following lemma follows from the same line of argument as Lemma~\ref{lem:first-layer-dominate}. We omit the details.
\begin{lemma}\label{lem:first-scale-dominate}
For large $R$, we have
\[\sum_{2\leq k\leq K_{R}}\mathbb{P}_{0}(G_{k})\leq R^{-99}.\]    
\end{lemma}
We are now ready to prove the extrinsic estimate of $\mathfrak{T}$.
\begin{prop}\label{prop:extrinsic-WUSF}
There exist $c_{1}>0$, and for any $\varepsilon>0$, $c_{2}=c_{2}(\varepsilon)>0$ such that for any $R>0$,
\[c_{1}R^{-1}\leq\mathbb{P}_{0}(\diam_{\mathrm{ext}}(\mathfrak{T})\geq R)\leq c_{2}R^{-1+\varepsilon}.\]
\end{prop}
\begin{proof}
The lower bound is a direct consequence of \eqref{eq:one-point-vWUSF}. For the upper bound, note that $\{\diam_{\mathrm{ext}}(\mathfrak{T})\geq R\}$ implies $E_{1}\cup\bigcup_{2\leq k\leq K_{R}}G_{k}$. By Lemma~\ref{lem:first-scale-dominate} and \eqref{eq:one-point-vWUSF}, we have
\[\mathbb{P}_{0}(\diam_{\mathrm{ext}}(\mathfrak{T})\geq R)\leq\mathbb{P}_{0}(E_{1})+R^{-99}\leq\sum_{x\in B_R^{c}\cap\Lambda_{1}}\|x\|^{-1}+R^{-99}\leq c(\varepsilon)R^{-1+6\varepsilon},\]
which completes the proof by the arbitrariness of $\varepsilon$.
\end{proof}
Similar to Theorems~\ref{thm:intrinsic} and \ref{thm:volume}, we can show the following intrinsic and volume estimate of $\mathfrak{T}$. Recall that $\beta=2-\alpha$ defined in \eqref{eq:betadef} stands for the growth exponent of 3D LERW.
\begin{prop}\label{prop:intrinsic-WUSF}
For any $\varepsilon>0$, there exist $c_1=c_1(\varepsilon),c_{2}=c_{2}(\varepsilon)>0$, such that for any $R>0$,
\[c_{1}R^{-1/\beta-\varepsilon}\leq\mathbb{P}_{0}(\diam_{\mathrm{int}}(\mathfrak{T})\geq R)\leq c_{2}R^{-1/\beta+\varepsilon}.\]
\end{prop}
\begin{prop}\label{prop:volume-WUSF}
For any $\varepsilon>0$, there exist $c_{1}=c_{1}(\varepsilon)>0$,\ $c_{2}=c_{2}(\varepsilon)>0$, such that for any $n>0$, 
\[c_{1}n^{-1/3-\varepsilon}\leq\mathbb{P}_{0}(|\mathfrak{T}|\geq n)\leq c_{2}n^{-1/3+\varepsilon}.\]
\end{prop}
The proofs of Propositions~\ref{prop:intrinsic-WUSF} and \ref{prop:volume-WUSF} are similar to those of Theorems~\ref{thm:intrinsic} and \ref{thm:volume} in the previous section. One only needs to replace Lemmas~\ref{lem:ext-int}--\ref{lem:int-vol} by their $0$-WUSF analogues, namely Lemmas~\ref{lem:ext-int-0}--\ref{lem:int-vol-0} from Section~\ref{subsec:ust-usf}. We remark that, as in the case of Theorems~\ref{thm:intrinsic} and \ref{thm:volume}, the $-
\varepsilon$ term can also be removed with some extra effort.
\section{Application to Abelian sandpile model}\label{sec:sandpile}
In this section, inspired by a similar argument from \cite{Hut20}, we derive the sandpile exponents from those of $0$-tree in the $0$-WUSF, obtained in Section \ref{sec:0wusf}.
\begin{proof}[Proof of Theorem~\ref{thm:sandpile-radius}.]
By the infinite-volume first-wave/$0$-tree correspondence, specifically \eqref{eq:first-wave-radius-tree}, and since $W_{1}(H)\subseteq \mathrm{AvC}$, Proposition~\ref{prop:extrinsic-WUSF} yields, for $R\ge1$,
\begin{align*}
\mathbf{P}\bigl(\diam_{\mathrm{ext}}(\mathrm{AvC})\geq R\bigr)
&\geq \mathbf{P}\bigl(\diam_{\mathrm{ext}}(W_{1}(H))\geq R\bigr)\\
&=G(0,0)\,\mathbb{P}_{0}\bigl(\diam_{\mathrm{ext}}(\mathfrak{T})\geq R\bigr)
\gtrsim R^{-1}.
\end{align*}
For the upper bound, \cite[Lemma 2.6]{BHJ17} (see also \cite[(9.3)]{Hut20}) gives
\[
\mathbf{P}\bigl(\diam_{\mathrm{ext}}(\mathrm{AvC})\geq R\bigr)\leq G(0,0)\,\mathbb{P}_{0}\bigl(\diam_{\mathrm{ext}}(\mathfrak{T})\geq R\bigr).
\]
Combining this with Proposition~\ref{prop:extrinsic-WUSF} completes the proof.
\end{proof}
\begin{proof}[Proof of Theorem~\ref{thm:sandpile-size}]

In a single wave, each vertex topples at most once. Therefore both $\mathrm{Av}$ and $|\mathrm{AvC}|$ dominate $|W_{1}(H)|$. By \eqref{eq:first-wave-volume-tree} and Proposition~\ref{prop:volume-WUSF}, for $n\ge2$,
\[
\mathbf{P}\bigl(|\mathrm{AvC}|\ge n\bigr)
\ge \mathbf{P}\bigl(|W_{1}(H)|\ge n\bigr)
=G(0,0)\,\mathbb{P}_{0}\bigl(|\mathfrak{T}|\ge n\bigr)
\gtrsim n^{-1/3-\varepsilon}.
\]
The case $n=1$ is absorbed by changing the constants.

We now prove the upper bound, which is achieved in a similar manner as in \cite[Section~9]{Hut20}. For every $m\geq1$,
\[
\mathbf{P}\bigl(\mathrm{Av}\geq n\bigr)
\leq \mathbf{P}\bigl(\diam_{\mathrm{ext}}(\mathrm{AvC})\geq m\bigr)
+\frac{1}{n}\,\mathbf{E}\Bigl[\sum_{x\in B_m}N(0,x)\Bigr].
\]
By \eqref{eq:Dhar-inf}, the expectation on the right-hand side is equal to $\sum_{x\in B_m}G(0,x)\lesssim m^{2}$ in dimension three. Therefore, by Theorem~\ref{thm:sandpile-radius},
\[
\mathbf{P}\bigl(\mathrm{Av}\geq n\bigr)\lesssim m^{-1+\varepsilon}+\frac{m^{2}}{n}.
\]
Choosing $m=\lfloor n^{1/3}\rfloor$ proves the claimed upper bound.
\end{proof}

\begin{rmk}\label{rem:2D}
Our method does not yield meaningful two-dimensional sandpile bounds. The first obstruction is conceptual: the three-dimensional argument is built on the transient $0$-WUSF picture, whereas in $d=2$ one must instead work with the recurrent $0$-UST. More seriously, the infinite-volume two-dimensional avalanche is substantially subtler: unlike in $d\geq 3$, one cannot simply appeal to almost-sure finiteness of avalanches and transfer tree estimates wave by wave in the same manner. In fact, it is not even known whether avalanches are almost surely finite in two dimensions. This is why the existing two-dimensional results rely on a different last-wave analysis, and why Table~\ref{tab:0wusf-sandpile-exponents} records only bounds on the exponents in that case.
\end{rmk}

\end{document}